\documentclass[10pt,notitlepage,twoside,a4paper]{amsart}
 \usepackage{amsfonts}

\usepackage{amsmath,amssymb,enumerate}

\usepackage{epsfig,fancyhdr,color}

\usepackage{amssymb}
\usepackage{amsmath,amsthm}
\usepackage{latexsym}
\usepackage{amscd}
\usepackage{psfrag}
\usepackage{graphicx}
\usepackage[latin1]{inputenc}
\usepackage[all]{xy}
\usepackage[mathcal]{eucal}

\definecolor{NoteColor}{rgb}{1,0,0}

\renewcommand{\textsc}{\textcolor{red}}

\newtheorem{theorem}{\rm\bf Theorem}[section]
\newtheorem{proposition}[theorem]{\rm\bf Proposition}
\newtheorem{lemma}[theorem]{\rm\bf Lemma}

\newtheorem*{theorem 1}{\rm\bf Proposition 1}
\newtheorem*{theorem 2}{\rm\bf Proposition 2}

\theoremstyle{definition}

\theoremstyle{remark}

\def\interieur#1{\mathord{\mathop{\kern 0pt #1}\limits^\circ}}

\title[Deforming Hexagons]
{Deforming hyperbolic hexagons with applications to the arc and the Thurston metrics on Teichm\"uller spaces}

\author{Athanase Papadopoulos}
\thanks{Athanase Papadopoulos was partially supported by the grant ANR-12-BS01-0009 (G\'eom\'etrie de Finsler et applications).}
\address{Athanase Papadopoulos,  Universit{\'e} de Strasbourg and CNRS,
7 rue Ren\'e Descartes,
 67084 Strasbourg CEDEX, France}
\email{athanase.papadopoulos@math.unistra.fr}

\author{Sumio Yamada }
\thanks{Sumio Yamada was partially supported by JSPS KAKENHI 24340009 and 16K13758}
 \address{Sumio Yamada, Gakushuin University, 1-5-1 Mejiro, Tokyo, 171-8588, Japan}
\email{yamada@math.gakushuin.ac.jp}
\date{\today}

\begin{document}

  \maketitle
  \begin{abstract}
  For each right-angled hexagon in the hyperbolic plane, we construct a one-parameter family of right-angled hexagons with a Lipschitz map between any two elements in this family,  realizing the smallest Lipschitz constant in the homotopy class of this map relative to the boundary.  
As a consequence of this construction, we exhibit new geodesics for the arc metric on the Teichm\"uller space of an arbitrary surface of negative Euler characteristic with nonempty  boundary. We also obtain new geodesics for Thurston's metric on Teichm\"uller spaces of hyperbolic surfaces without boundary. Our results generalize results obtained in the two papers \cite{PT1} and \cite{PT2}.
  \end{abstract}


\medskip

\medskip

\section{Introduction}\label{intro}

 The Teichm\"uller space of a surface admits several natural metrics, starting with the Teichm\"uller metric introduced by Teichm\"uller in 1939, followed by the Weil-Petersson metric introduced by Weil in 1958, and there are many others. It is reasonable to say that the next metric which is most actively investigated today is the one introduced by Thurston in 1985 or 1986, which carries now the name Thurston's metric. These three metrics, along with   others defined on Teichm\"uller space,  were studied from various points of view: the infinitesimal structure (Finsler or Riemannian), the geodesics, the convexity properties, the isometries, the quasi-isometries, the curvature properties, the boundary structure, etc.  Some basic questions concerning these properties were solved and others, more difficult, remain open and make the subject a living one.
 
In the present paper, the setting is the non-reduced Teichm\"uller theory of surfaces with boundary. We mean by this that the homotopies that we consider, in the definition of the equivalence relation defining Teichm\"uller space, do not necessarily fix  pointwise the boundary of the surface. In general, the metrics on Teichm\"uller space were defined for closed surfaces or surfaces with punctures (or distinguished points). Some of these metrics, like the Teichm\"uller metric, admit a straightforward generalization to the case of surfaces with boundary, but other metrics do not generalize as such, and one has to modify their definition to make them fit to surfaces with boundary. One example is Thurston's metric, whose modification, the so-called arc metric, was studied in the papers \cite{ALPS}, \cite{LPST1},  \cite{PT1}, \cite{PT2}. 
We shall recall the definition of the arc metric below.
Several basic questions concerning this metric still resist. For instance, it is unknown whether it is Finsler,  its isometry group is still not identified, and it is  unknown whether this metric coincides with the Lipschitz metric. (All these results are known for Thurston's metric.) In the present paper, we construct new families of geodesics for this metric. In order to state more precisely the results, we start with some notation.

Let $S$ be a surface of finite type with nonempty boundary.
The hyperbolic structures on $S$ that we consider are such that the boundary components are closed geodesics. 

 A simple closed curve on $S$  is \emph{essential} if it is neither homotopic to a point nor to a puncture (but it can be homotopic to a boundary component).

An \emph{arc} on $S$ is the image of a compact interval of $\mathbb{R}$ by a proper embedding, that is, the interior of the arc is embedded in the interior of $S$ and its  endpoints are on the boundary of $S$. In this paper, when we deal with  homotopies of arcs, we only consider homotopies that are relative to $\partial S$, that is, they keep the endpoints of the arc on the boundary of the surface (but they do not
necessarily fix pointwise the points of that boundary). An arc is \emph{essential} if it is not homotopic to an arc contained in $\partial S$.

We shall use the following notation:


  $\mathcal{B}$ is the union of the set of homotopy classes of essential arcs on $S$ and  the set of homotopy classes of simple closed curves homotopic to a boundary component of $S$. 

 $\mathcal{S}$ is the set of homotopy classes of essential simple closed curves on $S$.

Suppose that the surface $S$ is equipped with a hyperbolic structure $g$. In any homotopy class of essential arc in $S$, there is a unique geodesic arc whose length is minimal among the arcs in that class relative the boundary. This geodesic arc makes at its endpoints right angles with the boundary of $S$. 
 We denote by $\ell_\gamma(g)$ the length of this geodesic arc, and we call it the \emph{geodesic length} of $\gamma$ for the hyperbolic metric $g$.  Likewise, for any element $\gamma$ of $ \mathcal{S}$, we denote by 
$\ell_\gamma(g)$  the length of its unique geodesic representative for the hyperbolic metric $g$.

 Let $\mathcal{T}(S)$ be the Teichm\"uller space of $S$. We view $\mathcal{T}(S)$ as the space of homotopy classes of hyperbolic structures on $S$ with geodesic boundary, where the lengths of the boundary components are not fixed. The   \emph{arc metric} $\mathcal{A}$ on $\mathcal{T}(S)$ is defined by the formula
 \begin{equation} \label{eq:arc}
\mathcal{A} (g,h)= \sup_{\gamma\in \mathcal{B}\cup \mathcal{S}} \log \frac{\ell_\gamma(h)}{\ell_\gamma(g)} 
\end{equation}
where $g$ and $h$ are hyperbolic structures on $S$.
This is an asymmetric metric (that is, it satisfies all the axioms of a metric except the symmetry axiom). It was introduced in \cite{LPST} and it is an analogue for surfaces with boundary of Thurston's asymmetric metric on the Teichm\"uller space of a surface without boundary (possibly with cusps) \cite{Thurston}. It was shown in \cite{LPST}  (Proposition 2.13) that one obtains the same metric by using Formula (\ref{eq:arc}) except that the supremum is taken now over $\mathcal{B}$ instead of $ \mathcal{B}\cup \mathcal{S}$. 
In other words, we also have
\begin{equation} \label{eq:arc1}
\mathcal{A} (g,h)= \sup_{\gamma\in \mathcal{B}} \log \frac{\ell_\gamma(h)}{\ell_\gamma(g)} 
\end{equation}
The arc metric is studied in  the papers  \cite{ALPS}, \cite{LPST1},  \cite{PT1} and \cite{PT2}.

We now recall the definition of the Lipschitz metric on the Teichm\"uller space of a surface with boundary. The definition is the same as the one of the Lipschitz metric defined by Thurston on Teichm\"uller spaces of surfaces without boundary. 

One first defines the {\it Lipschitz constant} of a homeomorphism $f: (X,d_X)\to (Y,d_Y)$ between two metric spaces by the formula
\begin{equation}\label{Lip}
\hbox{Lip}(f)=\sup_{x\neq y\in X}\frac{d_{Y}\big{(}f(x),f(y)\big{)}}{d_{X}\big{(}x,y\big{)}}\in\mathbb{R}\cup\{\infty\}.
\end{equation}
The homeomorphism $f$ is said to be \emph{Lipschitz} if its Lipschitz constant is finite.

Given an ordered pair of hyperbolic structures $g$ and $h$ on $S$, the \emph{Lipschitz distance} between them (and between the corresponding points in the Teichm\"uller space $\mathcal{T}(S)$) is defined as
\begin{equation} \label{eq:Lip}
L(g,h)= \log \inf_{f\sim \mathrm{Id}_{S}} \hbox{Lip}(f)
\end{equation}
where the infimum is taken over all homeomorphisms $f:(S,g)\to (S,h)$ in the homotopy class of the identity of $S$.

In the case of surfaces without boundary, Thurston's metric and the Lipschitz metric coincide. This is a result of Thurston in \cite{Thurston}. It is unknown whether in the case of surfaces with boundary the Lipschitz and the arc metrics coincide.

In Thurston's theory for surfaces without boundary developed in \cite{Thurston}, maps between ideal triangles are the building blocks for the construction of geodesics for Thurston's metric on Teichm\"uller space. The geodesics obtained in this way are the so-called ``stretch lines" and Thurston proves that any two points in the Teichm\"uller space of $S$ are joined by a concatenation of finitely many stretch lines. 
 It is  possible to construct a class of geodesics for the arc metric $\mathcal{A}$ using Thurston's method for stretch lines. For this, a complete maximal geodesic lamination is needed (in the language of \cite{Thurston}, this will be the geodesic lamination which is maximally stretched); for instance, we can take a lamination whose leaves spiral along the  boundary components of the surface $S$. We then apply Thurston's method described in \cite{Thurston} for the construction of stretch lines using the stretch maps between ideal triangles and gluing them over all the ideal triangles that are the connected components of the complement of the lamination. To see that the one-parameter family of surfaces obtained in this way is a geodesic for the arc metric $\mathcal{A}$, one can double the surface $S$  along its boundary components and consider the resulting one-parameter family of hyperbolic structures on the doubled surface $S^d$. This is a geodesic for Thurston's metric on $S^d$. The restriction to $S$ of this one-parameter family of deformations of hyperbolic structures is a  geodesic for the metrics $\mathcal{A}$ and $L$ on the Teichm\"uller space of the surface with boundary. A geodesic for the arc metric defined in this way is of a special type (that is, not all geodesics for the arc metric are obtained in this manner). Along this geodesic, the lengths of all the geodesic boundary components of $S$ are multiplied by the same constant factor.

Another construction of $\mathcal{A}$-geodesics in the setting of surfaces with boundary is obtained by taking as building blocks right-angled hexagons instead of ideal triangles. In the present paper, we shall introduce canonical maps between such hexagons which depend on a parameter $K$. By varying $K$ we obtain geodesics for the arc metric. Each right-angled hexagon belongs to such a family of deformations, and between each pair of hexagons of the same family we exhibit a 
Lipschitz map that realizes  the best Lipschitz constant in its homotopy class relative to the boundary. Gluing the hexagons along pieces of their boundaries, we obtain deformations of the hyperbolic surfaces that are geodesics for the arc metric. These geodesics are also
geodesic for the Lipschitz metric and the arc and the Lipschitz metrics coincide on these 
lines. By gluing surfaces with boundary along their boundary components, we obtain geodesics for  Thurston's metric that are different from the stretch lines.
This generalizes the set of results obtained in the two papers \cite{PT1} and \cite{PT2}.

\section{Geometry of right-angled hexagons} \label{s:geo}

\subsection{Three types of right-angled hexagons}
For the construction of Lipschitz maps between arbitrary right-angled hexagons, we shall divide such a hexagon $H$ into three regions (one such region may possibly be empty) which carry natural coordinates; that is, the points in such a region are parametrized by pairs of real numbers. The coordinates in each region are induced by a pair of orthogonal foliations $F$ and $G$ on $H$, which we now define.

 \begin{figure}[!hbp] 
\centering
 \psfrag{s1}{\small $s_3$}
  \psfrag{s2}{\small $s_1$}
   \psfrag{s3}{\small $s_2$}
    \psfrag{l1}{\small $l_3$}
  \psfrag{l2}{\small $l_1$}
   \psfrag{l3}{\small $l_2$}
\psfrag{t1}{\small $t_3$}
  \psfrag{t2}{\small $t_1$}
   \psfrag{t3}{\small $t_2$}
\includegraphics[width=0.50\linewidth]{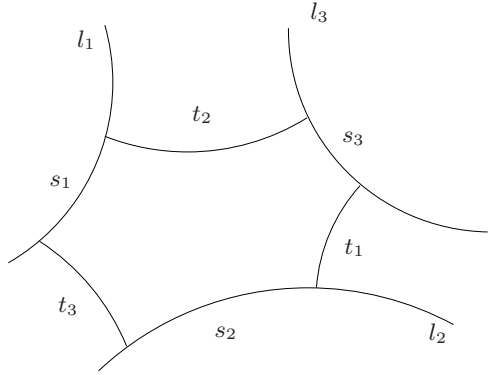}    
\caption{\small {Six geodesic lines enclosing a hexagon.}}   \label{hex}  
\end{figure}

We consider three pairwise non-consecutive edges $t_1, t_2, t_3$  of a right-angled hexagon $H$, and we call them the {\it short edges}. The three lengths $\lambda_1, \lambda_2, \lambda_3$ of $t_1, t_2, t_3$ may or may not satisfy the triangle inequality. Up to a permutation of the indices,  there are three cases:
\begin{enumerate}
\item Type I: $\lambda_1,\lambda_2,\lambda_3$ satisfy the three strict triangle inequalities.
\item Type II:  $\lambda_1+\lambda_2=\lambda_3$.
\item Type III: $\lambda_1+\lambda_2<\lambda_3$.
\end{enumerate}

We now give another characterization of the three distinct types.

 Recall that two geodesic lines in the hyperbolic plane are said to be hyper-parallel when they intersect each other neither in the hyperbolic plane, nor on the geometric boundary of this plane. In the case where the geodesics intersect each other on the boundary, the geodesics are said to be \emph{asymptotic}, or \emph{parallel}. 
 
 Consider three pairwise hyper-parallel geodesic lines $l_1, l_2, l_3$ in the hyperbolic plane, relatively positioned so that for each geodesic $l_i$, the other two geodesics $l_j, l_k$,  $(j, k \neq i)$ lie on the same side of $l_i$. In particular, we are excluding the possibility that some $l_i$ and $l_j$ are asymptotic to each other. 
Each pair $l_i, l_j$ $(i \neq j)$  has a common perpendicular geodesic segment 
$t_k$ $(k \neq i, j)$. This common perpendicular segment, which we call a \emph{short  edge}, is unique.  We denote by $\tilde{l}_1, \tilde{l}_2, \tilde{l}_3$, the three geodesic lines that contain the short edges  $t_1, t_2, t_3$ respectively.  The six geodesic lines $l_1, l_2, l_3, \tilde{l}_1, \tilde{l}_2, \tilde{l}_3$ enclose a hexagon with six right angles (Figure \ref{hex}).  The side opposite to a short edge $t_i$ will be denoted by $s_i \subset l_i$, and  called a \emph{long edge}. 

There exists a unique point $O$ which is equidistant from $\tilde{l}_1, \tilde{l}_2, \tilde{l}_3$. The three types listed above correspond respectively to the cases where the point $O$ lies inside the right-angled hexagon $H$ (Type I), or on one 
of the long edges $s_i$ (Type II), or outside $H$ (Type III); cf. Figure \ref{center}.

One should be aware of the fact that interchanging the short and long edges of a given right-angled hexagon makes the resulting center $O$ different from the original 
center unless the hexagon has a ${\mathbb Z}_3$-rotational symmetry.

 From now on, when we refer to a right-angled hexagon, it will be understood that  a choice of short and long edges has been made.

\subsection{The tripod and the hypercycle foliation of a right-angled hexagon}
A right-angled hexagon $H$ is naturally equipped with 
a measured foliation  which is a union of three foliated regions $F_1, F_2, F_3$ with disjoint interiors (one of them may be empty). We explain below the construction of this foliation.

For each $i=1,2,3$, we  consider the orthogonal projection of $O$ onto the geodesic $\tilde{l}_i \supset t_i$, which we denote by $A_i$. We obtain a tripod $T$ whose edges $OA_1, OA_2$ and $OA_3$ have equal lengths $d>0$, and where each edge $OA_i$ meets $\tilde{l}_i$ perpendicularly. In Figure \ref{center}, he have represented the tripod in the three cases of hexagons (Types I, II and III). 
 \begin{figure}[!hbp] 
\centering
   \psfrag{s3}{\small $s_1$}
    \psfrag{l1}{\small $\tilde{l}_2$}
  \psfrag{l2}{\small $\tilde{l}_3$}
   \psfrag{l3}{\small $\tilde{l}_1$}
    \psfrag{A3}{\small $A_1$}
     \psfrag{O}{\small $O$}
\includegraphics[width=1\linewidth]{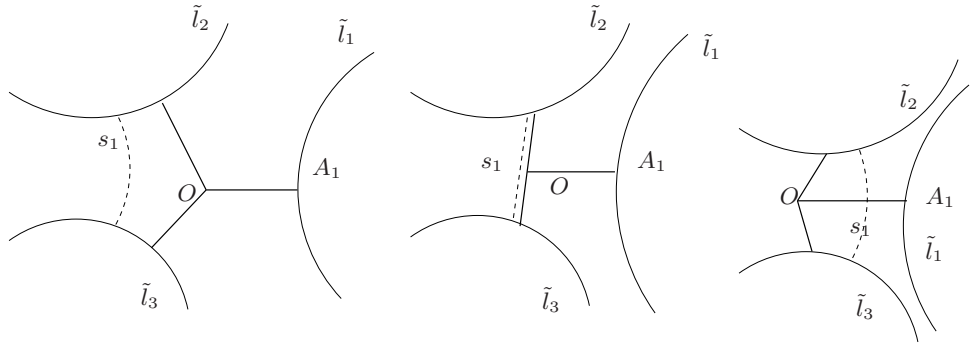}    
\caption{\small {The tripod is drawn in bold lines. The center $O$ of the hexagon may be in the interior of the hexagon (figure to the left, Type I), or on one side (the side $s_1$ in the figure in the middle, Type II), or outside the hexagon (figure to the right, Type III). In the case in the middle, the side $s_1$ (dashed line) coincides with the union of two edges of the tripod.}}   \label{center}  
\end{figure}

Choose one pair among the three tripod edges, say, $OA_1$ and $OA_2$.  The pair  $OA_1$ and $OA_2$ and the long edge $s_3$ together with (part of) the short edges $t_1$ and $t_2$ bounds a region, which is a pentagon $P_3$ with four right-angled corners (Figure \ref{penta1}). 
 \begin{figure}[!hbp] 
\centering
 \psfrag{t1}{\small $t_1$}
  \psfrag{t2}{\small $t_2$}
   \psfrag{A1}{\small $A_1$}
    \psfrag{A2}{\small $A_2$}
  \psfrag{s3}{\small $s_3$}
   \psfrag{P3}{\small $P_3$}
     \psfrag{O}{\small $O$}
\includegraphics[width=0.55\linewidth]{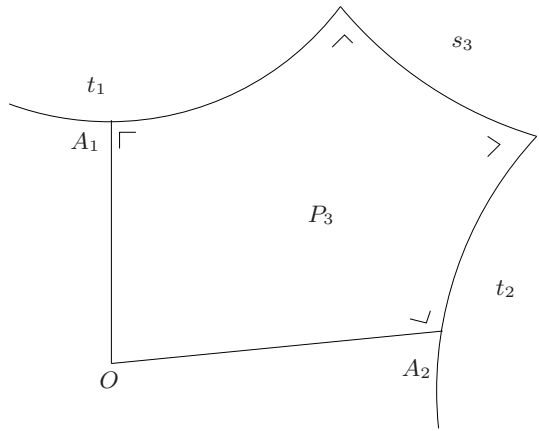}    
\caption{\small {The pentagon $P_3$}}   
\label{penta1}  
\end{figure}

Draw in $P_3$ a family of hypercycles  each of which is a locus of equidistant points measured from the long edge $s_3$. In the hyperbolic plane, equidistant points to geodesics are classically  called hypercycles, and we shall use this terminology. $P_1$ and $P_2$ are similarly defined.

For each $i=1,2,3$,  $P_i$ consists of two regions $R_i$ and $C_i$, where the former is foliated by hypercycles whose endpoints are on the long edge $t_j$ and $t_k$ for $j, k \neq i$, and the latter is foliated by  hypercycles whose endpoints are on the tripod edges $OA_j$ and $OA_k$.  The foliation of $P_i$ equipped with its foliation $F_i$ is represented in Figure \ref{penta2}.
\begin{figure}[!hbp] 
\centering
 \psfrag{t1}{\small $t_1$}
  \psfrag{t2}{\small $t_2$}
   \psfrag{Aj}{\small $A_j$}
    \psfrag{Ak}{\small $A_k$}
  \psfrag{Ci}{\small $C_i$}
   \psfrag{Ri}{\small $R_i$}
     \psfrag{O}{\small $O$}
\includegraphics[width=0.45\linewidth]{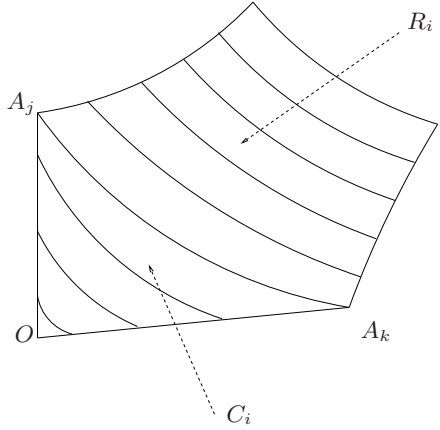}    
\caption{\small {The pentagon $P_i$ equipped with its foliation $F_i$.}}   
\label{penta2}  
\end{figure}

We consider the region 
\[
\tilde{H} := P_1 \cup P_2 \cup P_3 
\]
For hexagons of Type I and II, $\tilde{H} $ coincides with the right-angled hexagon  $H$, and for Type III, $\tilde{H} $ strictly contains $H$. The leaves of the three foliations $F_1, F_2, F_3$ may possibly intersect in $\tilde{H}$. Restrict the three foliations $F_1, F_2, F_3$ to the original hexagon $H$ to obtain a foliation of $H$.

We call the union of three regions $(\cup_i C_i) \cap H$ the \emph{central region} $C$ of $H$.   
 \begin{figure}[!hbp] 
\centering
 \psfrag{1}{\small $P_1$}
  \psfrag{2}{\small $P_2$}
   \psfrag{3}{\small $P_3$}
   \psfrag{A1}{\small $A_1$}
  \psfrag{A2}{\small $A_2$}
   \psfrag{A3}{\small $A_3$}
   \psfrag{s1}{\small $s_1$}
  \psfrag{s2}{\small $s_2$}
   \psfrag{s3}{\small $s_3$}
\includegraphics[width=0.65\linewidth]{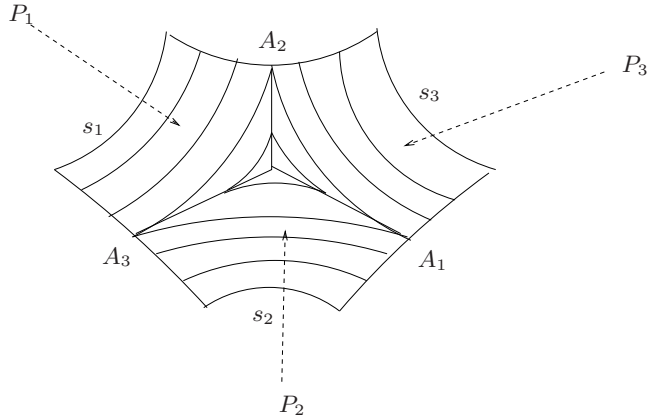}    
\caption{\small {The hexagon $H$ is the union of the three pentagons $P_1,P_2,P_3$.}}   \label{penta}  
\end{figure}

The supports of the foliations $F_i$ by hypercycles restricted to $R_i \subset P_i$ have a natural structure of rectangles and these foliations are  equipped with transverse measures induced from the Lebesgue measure on the boundaries of the supports, whose total masses $L_1,L_2,L_3$ satisfy the equations
 \[
  \displaystyle \begin{cases} \displaystyle
 \lambda_1=L_2+L_3\\
 \lambda_2=L_1+L_3\\
\lambda_3=L_1+L_2.
  \end{cases}
\]
with an appropriate meaning when one of the lengths $L_i$ is negative, as we discuss now. This is represented in Figure \ref{negative} below.

For Type I, the central region $C$ is bounded by the three hypercycles distant from $s_i$ by $L_i> 0$ for each $i$ and $C = C_1 \cup C_2 \cup C_3$. Types II and III correspond respectively  to the cases when one of the total masses $L_i$, say $L_3$,  is zero or strictly negative, and the central region is bounded by two hypercycles and one boundary geodesic segment $s_3$. In Types II and III, when $L_3 \leq 0$ (or equivalently when the set ${\rm interior}(F_3) \cap H$ 
is empty) the strict triangle inequality
\[
\lambda_1 + \lambda_2 > \lambda_3
\]
among the lengths of the short sides is not satisfied. The three types of foliations corresponding to Types I, II and III are represented in Figure \ref{foliation}.
 \begin{figure}[!hbp] 
\centering
 \psfrag{s1}{\small $s_1$}
  \psfrag{s2}{\small $s_2$}
   \psfrag{s3}{\small $s_3$}  
     \psfrag{t1}{\small $t_1$}
  \psfrag{t2}{\small $t_2$}
   \psfrag{t3}{\small $t_3$}
     \psfrag{O}{\small $O$}
      \psfrag{T1}{\small Type I}
       \psfrag{T2}{\small Type II}
        \psfrag{T3}{\small Type III}
\includegraphics[width=0.750\linewidth]{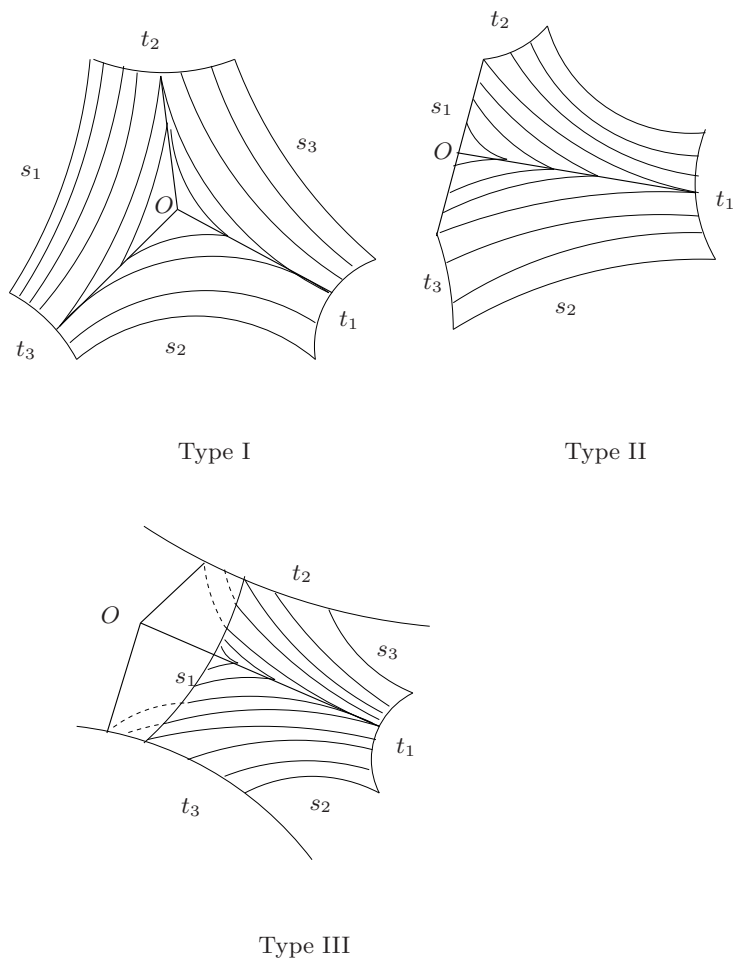}    
\caption{\small {The three types of hypercycle foliations of a right-angled hexagon.}}   \label{foliation}  
\end{figure}

In Type I, corresponding to the case to the left picture of Figures \ref{center} and \ref{foliation},  for each $i =1,2,3$, $A_i$ divides each short edge $t_i$ at an inner point. In this case  $L_1,L_2,L_3$ are all positive.   In Type II, two of the points $A_i$, say $A_2$ and $A_3$, coincide with endpoints of $t_2$ and $t_3$ respectively. In this case, we have $L_1=0$ and the foliation $H$ is of the type depicted on the upper right hexagon of Figure \ref{foliation}.
Finally, Type III corresponds to the case where two of the $A_i$, say $A_2$ and $A_3$,  divide the short edges $t_2$ and $t_3$ externally. In this case  we have $L_1 < 0$, and the foliation $H$ is as depicted in the bottom hexagon of Figure \ref{foliation}, and in Figure \ref{negative}.  

%
%
%
%
%
%
%
%

We now look into the Type I case more closely using the Poincar\'e disc model and taking the point $O$ at the center of this disc. 
 \begin{figure}[!hbp] 
\centering
 \psfrag{L1}{\small $L_1<0$}
  \psfrag{L2}{\small $L_2$}
   \psfrag{L3}{\small $L_3$}
   \psfrag{O}{\small $O$}
    \psfrag{2}{\small $\alpha_2$}
     \psfrag{3}{\small $\alpha_3$}
     \psfrag{A1}{\small $A_1$}
  \psfrag{A2}{\small $A_2$}
   \psfrag{A3}{\small $A_3$}
\includegraphics[width=0.80\linewidth]{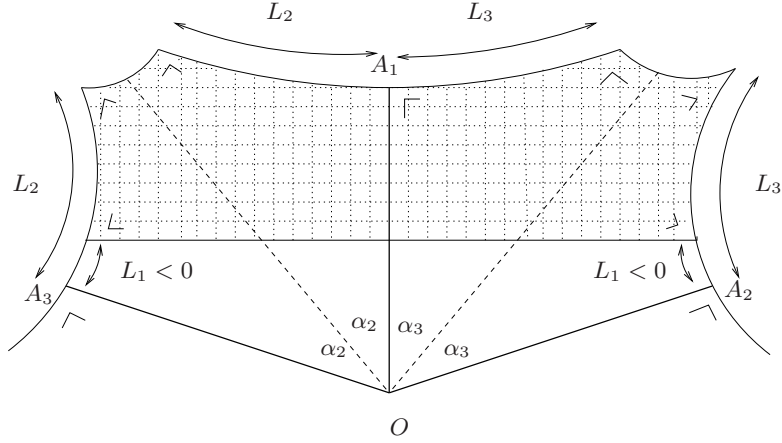}    
\caption{\small {The shaded region represents a Type III hexagon. The tripod is represented in bold lines. The point $O$ lies outside $H$.}}   \label{negative}  
\end{figure}
The hexagon $H$ is divided into three pentagons $P_1, P_2, P_3 $ by the geodesic tripod $T= \cup_{i=3} {\overline{OA_i}}$ whose edges have the same hyperbolic length $d$, each edge 
meeting one of the short edges perpendicularly.   Each pentagon $P_i$ is foliated by pieces of hypercycles which are equidistant from the long edges $s_i$ of $H$, and the hexagon $H$ is foliated as a union of three 
foliations $F_i$  $(i = 1,2, 3)$ (Figure \ref{penta}).   The pentagon $P_i$ has a ${\mathbb Z}_2$-symmetry across the angle bisector of $\angle {A_jO A_k}$ $(j, k \neq i)$. 

For each $i=1,2,3$, the value of the angle between $OA_j$ and $OA_k$ ($j,k\not=i$) is set to be equal to $2\alpha_i$. (The half-values, $\alpha_i$, will play a role below.)

The Type II hexagon is a limiting case of the Type I hexagon, where the two edges $\overline{OA_2}$ and $\overline{OA_3}$ form an angle $2 \alpha_1 = \pi$. In this case, two edges of the tripod are aligned, the pentagon $P_1$ is collapsed to the long edge $s_1$, and the hexagon is covered by the remaining two foliated regions. In other words, $H$ is a union of two pentagons $P_2$ and $P_3$ whose central vertex angles at $O$ satisfy $2\alpha_2 +2 \alpha_3=\pi$.

In the Type III hexagon, represented in Figure \ref{negative}, the region $H$ is a proper subset of two pentagons $P_2$ and $P_3$ (these are the left and right halves of the hexagon pictured) with its central angles  satisfying $2\alpha_2 + 2\alpha_3 < \pi$, or equivalently $2\alpha_1 > \pi$.   

 In what follows, we shall concentrate the discussion on Type I 
hexagons. Type II and III hexagons will be 
considered when needed.

\subsection{Orthogonal coordinate system on right-angled hexagons} 
We equip each pentagon $P_i$ with a second foliation $G_i$ transverse to $F_i$ whose leaves are the fibers of the nearest point projection map $\pi_i: P_i \rightarrow s_i$.  It turns out that the leaves of this foliation are geodesics which intersect orthogonally the leaves of $F$.

 In a Type I hexagon, we parameterize the leaves of the foliation $F_i$ by $0\leq  u_i \leq 2$ as follows. 
 
 For each $i$, the region $P_i$ is foliated 
by leaves $\{F_i(u_i)\}$ where $F_i(0)$ is the long edge $s_i$, $F_i(1)$ is the side of the central region in $P_i$, and $F_i(2)$ is the origin  $O$ (a degenerate leaf).  In 
between, for $0< u_i<1$, 
we interpolate the leaves proportionally to the hyperbolic distance from the long edge $s_i$. For $1< u_i <2$, the interpolation is done proportionally to the 
hyperbolic length along the edges of the geodesic tripod $T$.  

We parametrize the leaves of $G_i$ by $0 \leq v_i \leq 2$ such that $G_i(0)$ and $G_i(2)$ are parts of the short edges $t_j, t_k$ with $j, k \neq i$ of length $L_i$ sandwiching the long edge  $s_i$, and such that $G_i(1)$  is the geodesic segment  from the origin $O$ meeting the long edge $s_i$ perpendicularly at its midpoint and bisecting the central angle $2 \alpha_i$ at $O$.  This geodesic segment is shared by the two congruent quadrilaterals  $Q_i$ and $\hat{Q}_i$ with $P_i = Q_i \cup \hat{Q}_i$.  In between, the interpolation is done proportionally 
to hyperbolic length along the long edge $s_i$. 

With this pair of foliations $\{(F_i, G_i)\}_{i=1,2,3}$  where  the leaves are parameterized by $u_i$ and  $v_i$, the hexagon $P_i$ has a coordinate system. Namely each point $ p \in P_i \subset H$ is identified with an ordered pair $(u_i, v_i)$.  Note that 
along the gluing edges of the three pentagons $P_1, P_2$ and $P_3$, the parameters $u_i$ and $u_j$ with $i \neq j$ are compatible, that is,  $u_i(p)=u_j(p)$ if $p$ lies on the edge of the tripod between $P_i$ and $P_j$.  

A Type II and III hexagon $H$ has a similar coordinate system, which is just the restriction of the three pentagons $P_i$ to the hexagon $H$.

\section{Geometry of hyperbolic quadrilaterals with three right angles}\label{s:hyp}

A congruent pair of hyperbolic geodesic quadrilaterals with three right angles with opposite orientations  form a geodesic pentagon with four right angles (see Figure \ref{QQ} below). We study in this section  hyperbolic quadrilaterals with three right angles
We recall a set of classical formulae (see e.g. \cite{Fenchel}) for such quadrilaterals, relating the side lengths and the non-right angle $\alpha < \pi/2$ of  the quadrilateral. Such a quadrilateral is usually called a trirectangular quadrilateral.

 We represent the quadrilateral in the Poincar\'e disc model of the hyperbolic plane. We assume that $\alpha$ is positioned at the origin $O$, with two edges being radial segments, one of which (the side $OA$) having length $d$ and the other one (the side $OC$) having length $L + h$, where $L$ is the length of the edge $AB$ opposite to the side $OC$ (see Figure \ref{nh2}).  Finally, we denote the length of the side $BC$ by $\ell$. This quadrilateral is uniquely determined by the two values $\alpha$ and $d$. The trigonometric  formulae are the following:
\begin{eqnarray}
\label{e1}\frac{\cosh \ell}{\sin \alpha}  & = & \frac{\cosh d}{\sin \frac{\pi}{2}}   =  \frac{\sinh (L+h)}{\sinh L} \\
\label{e2}\sinh \ell & = &  \sinh d \cosh (L+h) - \cosh d \sinh (L + h) \cos \alpha \\
\label{e3}\sinh d & = & \sinh \ell  \cosh (L+h) - \cosh \ell \sinh (L + h)  \cos \frac{\pi}{2} \\
\label{e4} \cos \alpha & = & \sin \frac{\pi}{2} \sinh L \sinh \ell - \cos \frac{\pi}{2} \cosh L \\
\label{e5} \cos \frac{\pi}{2} & = & \sin \alpha \sinh L \sinh d - \cos \alpha \cosh L  \\
\label{e6} \cosh (L+h) & = & - \sinh \ell \sinh d + \cosh  \ell \cosh d  \cosh L \\
\label{e7} \cosh L & = & \sin \alpha \cosh (L+h)  
\end{eqnarray} 

In these formulae, the angle $\pi/2$ is the right angle at the vertex $C$.  In the book \cite{Fenchel}, the above formulae are given for a quadrilateral with two consecutive right angles, with the angle $C$ not necessarily $\pi/2$. We have rewritten them in the special case $C=\pi/2$. 

 \begin{figure}[!hbp] 
\centering
 \psfrag{L}{\small $L$}
  \psfrag{C}{\small $C$}
   \psfrag{B}{\small $B$}
  \psfrag{d}{\small $d$} 
 \psfrag{l}{\small $\ell$} 
  \psfrag{O}{\small $O$}
  \psfrag{P}{\small $P$} 
 \psfrag{a}{\small $\alpha$}
 \psfrag{M}{\small $M$}  
  \psfrag{h}{\small $h$}  
   \psfrag{A}{\small $A$} 
\includegraphics[width=0.45\linewidth]{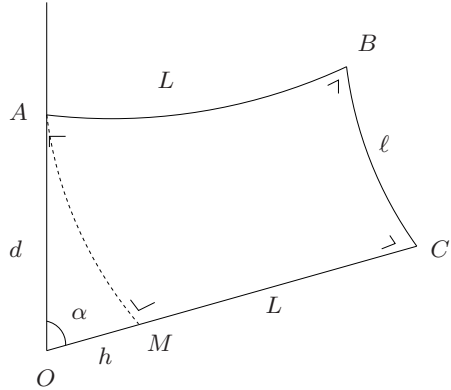}    
\caption{\small {The dashed line $AM$ is a piece of hypercycle parallel to the side $BC$. All the other lines are geodesics.}}   \label{nh2}  
\end{figure}

Now suppose that the value of the acute angle is $\pi - \alpha$ with $\alpha > \pi/2$.  The set of formulae then becomes
\begin{eqnarray*}
\frac{\cosh \ell}{\sin (\pi - \alpha)}  & = & \frac{\cosh d}{\sin \frac{\pi}{2}}   =  \frac{\sinh (L+h)}{\sinh L} \\
\sinh \ell & = &  \sinh d \cosh (L+h) - \cosh d \sinh (L + h) \cos (\pi -\alpha) \\
\sinh d & = & \sinh \ell  \cosh (L+h) - \cosh \ell \sinh (L + h) \cos \frac{\pi}{2} \\
\cos (\pi - \alpha) & = & \sin \frac{\pi}{2} \sinh L \sinh \ell - \cos \frac{\pi}{2} \cosh L \\
\cos \frac{\pi}{2} & = & \sin (\pi -\alpha) \sinh L \sinh d - \cos (\pi - \alpha) \cosh L  \\
\cosh (L+h) & = & - \sinh \ell \sinh d + \cosh  \ell \cosh d  \cosh L \\
\cosh L & = & \sin (\pi - \alpha) \cosh (L+h).  
\end{eqnarray*} 
By noting that $\cos(\pi - \alpha) = - \cos \alpha$, $\sin (\pi - \alpha) = \sin  \alpha$, $\cosh (-x) = \cosh x$ and $\sinh(-x) = - \sinh x$, we can rewrite these formulae as
\begin{eqnarray}
\label{e1.1}\frac{\cosh \ell}{\sin \alpha}  & = & \frac{\cosh d}{\sin \frac{\pi}{2}}   =  \frac{\sinh [-(L+h)]}{\sinh (-L)} \\
\label{e2.1}\sinh \ell & = &  \sinh d \cosh [-(L+h)] - \cosh d \sinh [-(L + h)] \cos \alpha \\
\label{e3.1}\sinh d & = & \sinh \ell  \cosh [-(L+h)] - \cosh \ell \sinh [-(L + h)]  \cos \frac{\pi}{2} \\
\label{e4.1} \cos \alpha & = & \sin \frac{\pi}{2} \sinh (-L) \sinh \ell - \cos \frac{\pi}{2} \cosh (-L) \\
\label{e5.1} \cos \frac{\pi}{2} & = & \sin \alpha \sinh (-L) \sinh d - \cos \alpha \cosh (-L)  \\
\label{e6.1} \cosh [-(L+h)] & = & - \sinh \ell \sinh d + \cosh  \ell \cosh d  \cosh (-L) \\
\label{e7.1} \cosh (-L) & = & \sin \alpha \cosh [-(L+h)].  
\end{eqnarray} 
Formally, in comparing them with equations (\ref{e1}, \ref{e2}, \ref{e3}, \ref{e4}, \ref{e5}, \ref{e6}, \ref{e7}),  these formulae describe the shape of a quadrilateral with three right angles and  one obtuse 
angle $\alpha > \pi/2$, and negative lengths $-L$ and $-h$, even though there is no such hyperbolic 
quadrilateral since the non-right-angle of a trirectangular quadrilateral in the hyperbolic plane is always acute, and there are no {\it negative} side lengths. These algebraic expressions, however, can be interpreted as follows.
     
Since the isometry type of a trirectangular quadrilateral is uniquely determined  (up to orientation) by $0< \alpha < \pi/2$ and the side length $d>0$, 
 consider  for a fixed $d$ the family of quadrilaterals obtained by increasing $\alpha$.   When $\alpha$ approaches $\pi/2$, 
 the side $BC$ of length $\ell$ converges to the side $OA$, and the sides $AB$ and $OC$ collapse to the points $A$ and $O$ respectively. Thus, the quadrilateral becomes degenerate and we have $\ell= d$ and $L = h =0$ in  Equations (\ref{e1}), (\ref{e2}) and (\ref{e3}).     

For $\alpha > \pi/2$ with the same $d>0$, we identify the situation given by Equations  (\ref{e1.1}, \ref{e2.1}, \ref{e3.1}, \ref{e4.1}, \ref{e5.1}, \ref{e6.1}, \ref{e7.1}) with a quadrilateral with vertex angle $\pi - \alpha$ appearing on the other side of $OA$,  the  image  of the corresponding quadrilateral with the angle $\alpha - \pi/2$ across the side $OA$.   The negative lengths $-L$ and $-h$ refer to the reflective symmetry across the side $OA$.  The situation is represented in Figure \ref{New}.

In this sense, the set of equations (\ref{e1}, \ref{e2}, \ref{e3}, \ref{e4}, \ref{e5}, \ref{e6}, \ref{e7}) describes the moduli of quadrilaterals with a fixed $d>0$ and variable vertex angle $\alpha$, where $\alpha$ varies 
in $[0, \pi]$, where the corresponding quadrilaterals are possibly degenerate, and with the opposite orientations with respect to the symmetry across the side $OA$.  
\begin{figure}[htbp]
\centering
 \psfrag{pi-}{\small $\pi-\alpha$}
  \psfrag{O}{\small $O$}
    \psfrag{A}{\small $A$}
   \psfrag{2a}{\small $2\alpha$}
  \psfrag{h}{\small $h$} 
 \psfrag{-h}{\small $-h$} 
 \psfrag{d}{\small $d$}
  \psfrag{-L}{\small $-L$}
   \psfrag{L}{\small $L$}
  \includegraphics[width=7cm]{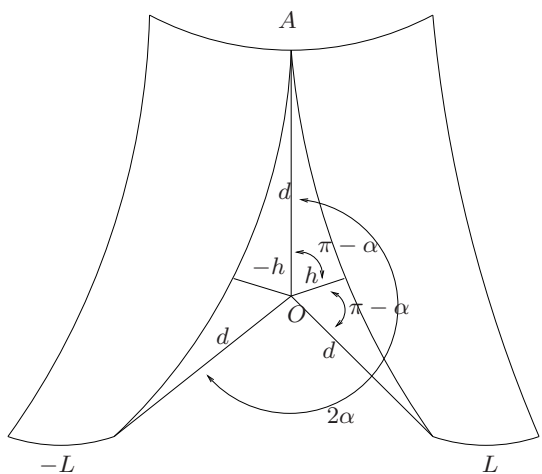}
\caption{\small{The trigonometric formulae (\ref{e1.1})--(\ref{e7.1}) refer to the four right-angled pentagon on the left hand side. In this pentagon, the value of the angle at $O$ is $2(\pi-\alpha)$, corresponding to the value $2\alpha$ for the exterior angle.}} \label{New}
\end{figure}

  %
%
%

\section{On the geometry of hexagons modeled on the Poincar\'e disc}

For a given trirectangular quadrilateral  $Q = OABC$ with three right angles at the vertices $A, B, C$ and one acute angle $\alpha$  at $O$, we consider the trirectangular  quadrilateral $\hat{Q}$ which is the mirror image of $Q$ by the
reflection through the straight edge  $OC$. We denote the corresponding vertices of $\hat{Q}$ by $\hat{A}, 
\hat{B}$, $\hat{O} = O$ and $\hat{C}=C$. The union of $Q$ and $\hat{Q}$ is a 
pentagon $P = OAB \hat{B} \hat{A}$ (Figure \ref{QQ}), where all the interior 
angles  are right, except at the vertex $O$ whose interior angle is $2 \alpha$. 

Now consider the situation where there are three pentagons $P_1, P_2$ and $P_3$ satisfying the 
compatibility 
condition $\alpha_1+\alpha_2+ \alpha_3 = \pi$ with $\alpha_i < \pi/2$ for each $i$, and $d_1=d_2=d_3$, a common value which we denote by $d$.  The three pentagons can 
be glued together via the identifications of the edges emanating from the origin $O$,
\begin{eqnarray*}
OA_1 & \simeq & O\hat{A}_2 \\
OA_2 & \simeq & O\hat{A}_3 \\
OA_3 & \simeq  &  O\hat{A}_1 \\
\end{eqnarray*}
to produce a 
right-angled hexagon $H$.  This is the Type I case described before, and it is depicted in Figure \ref{nh1}.  

The condition $\alpha_i < \pi/2$ is violated when $\alpha_1 = \pi/2$ and thus $\alpha_2 + \alpha_3 = \pi/2$.  This is the case when $P_1$ degenerates to the line segment $A_2O \cup OA_3$, and the hexagon is just $P_2 \cup P_3$.  

Finally when $\alpha_1 > \pi/2$, then $\pi - \alpha_1 < \pi/2$, and as described in the previous section, the quadrilateral $Q_1$ with its vertex angle $\pi - \alpha$ appears on the other side of $OA_2$ and $\hat{Q}_1$ on the other side of $OA_3$. Consequently  the pentagon $P_1 = Q_1 \cup \hat{Q}_1$ overlaps with $P_2$ and $P_3$. This is the type III picture. Note that the compatibility condition among $\alpha_1, \alpha_2, \alpha_3$ is still intact since the vertex angle of $P_1$ at $O$ is the sum of those of $P_2$ and $P_3$:
\[
\pi-\alpha_1 = \alpha_2 + \alpha_3.
\]     
Also note that by using the original $\alpha_1 > \pi/2$ in the set of equations (\ref{e1}, \ref{e2}, \ref{e3}, \ref{e4}, \ref{e5}, \ref{e6}, \ref{e7}), the resulting negativity of $L_1$ can now be understood with the observation we made in the case of Type III, in which the triangle inequality among  $\lambda_1, \lambda_2$ and $\lambda_3$ is violated.

\begin{figure}[htbp]
\centering
\psfrag{A}{\small $A$}
\psfrag{B}{\small $B$}
\psfrag{C}{\small $C=\hat{C}$}
\psfrag{O}{\small $O$}
\psfrag{A1}{\small $\hat{A}$}
\psfrag{B1}{\small $\hat{B}$}
\psfrag{Q}{\small $Q$}
\psfrag{Q1}{\small $\hat{Q}$}
\psfrag{l}{\small $\ell$}
\psfrag{d}{\small $d$}
\psfrag{L}{\small $L$}
\psfrag{L1}{\small $L+h$}
\psfrag{a}{\small $\alpha$}
\includegraphics[width=8cm]{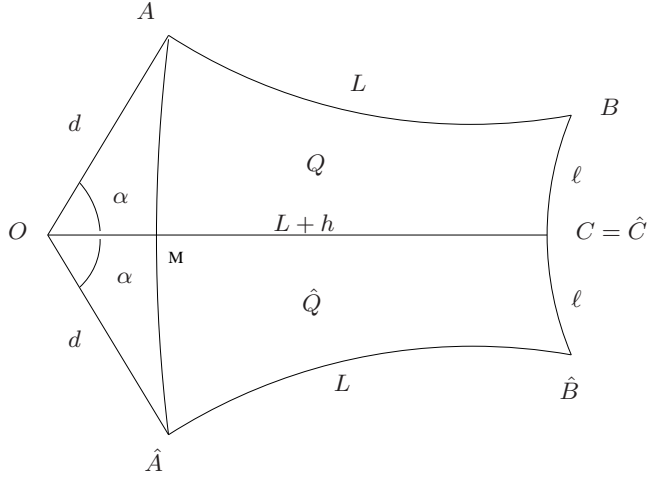}
\caption{\small{The pentagon is a union of two trirectangular quadrilaterals $Q$ and $\hat{Q}$. In this figure, the lines $AM$ and $\hat{A}M$ are hypercycles. All the other lines are geodesics.}} \label{QQ}
\end{figure}

We shall use the Poincar\'e unit disc model of the hyperbolic plane. We shall use at the same time the Euclidean and the hyperbolic geometry of this disc.

With the notation of Figure \ref{nh2}, with $s$ being the Euclidean distance from $O$ to $A$  in the Poincar\'e disc, and 
$t$ being the Euclidean distance  from the center $O$ to the point $M$, we can write $d$ and $h$ as functions of $s$ and $t$:
\[  d(s) =\log  \frac{1+s}{1-s}  \mbox{ and }
h(t) = \log  \frac{1+t}{1-t}.\]

\begin{proposition}
Using the above notation, we have the following formulae for the quadrilaterals $OABC$ and $BCMA$ in terms of the Euclidean distance parameters $s$ and $t$ (Figure \ref{QQ}): 
 \[
   \tanh L =  \frac{\cos \alpha}{\sin \alpha}\frac{1-s^2}{2s} \quad \mbox{ and  } \quad 
  t= \frac{\cos \alpha}{\sin \alpha +1}s
\]
\end{proposition}

\begin{proof}
We have the following relations:
\[
\cosh d = \frac{1+s^2}{1-s^2} \,\,\, \mbox{ and  } \,\,\, \sinh d = \frac{2s}{1-s^2} 
\]
and
\[
\cosh h = \frac{1+t^2}{1-t^2} \,\,\, \mbox{ and  } \,\,\, \sinh h = \frac{2t}{1-t^2}. 
\]

Equation (\ref{e7}) can be written as 
\[
\frac{\cosh (L+h)}{\cosh L} = \frac{1}{\sinh \alpha}
\]
By applying the angle-addition formula, the equality becomes
\[
\cosh h + \tanh L \sinh h = \frac{1}{\sin \alpha}.
\]
On the other hand, Equation (\ref{e5}) shows that
\[
\sinh d = \frac{1} {\tan \alpha \tanh L}
\]
which in turn says
\[
\tanh L = \frac{\cos \alpha}{\sin \alpha} \frac{1-s^2}{2s}.
\]
By combining these equations, we have the following relation satisfied by the two Euclidean parameters $s$ and $t$:
\begin{equation}\label{e8}
\frac{1+t^2}{1-t^2} + \frac{\cos \alpha}{\sin \alpha} \frac{1-s^2}{2s} \frac{2t}{1-t^2} = \frac{1}{\sin \alpha} 
\end{equation}

Factorizing in Equation (\ref{e8}),  we get 
\begin{equation}\label{FFFF}
\left(\frac{1+\sin \alpha}{\cos\alpha}t-s\right)\left(\frac{1+\sin\alpha}{\cos\alpha}ts+1\right)=0
\end{equation}
which in turn implies 
\begin{equation}\label{f:combining}
t = \frac{\cos \alpha}{1 +  \sin \alpha} s.
\end{equation}
We are using the fact that the second factor in \ref{FFFF} is always $>0$, since $t>0$, and $s\cos\alpha >0$.
\end{proof}

Formula (\ref{f:combining}), which states that for $\alpha$ constant, the two parameters $t$ and $s$ are linearly related, has the following geometric interpretation.  
Consider the arc of hypercycle which is the set of points at hyperbolic distance $L$ from the side $BC$. This arc starts at the vertex $A$ and  meets the side $OC$ at a 
point, which we call $M$, which is hyperbolic  distance $h$ from $O$.  The geometry of the Poincar\'e disc implies that the arc is a Euclidean circular arc, and we call the region surrounded by the line segment $OA$, the arc $AM$ and the side $MO$, the central region of the quadrilateral.  Note  when $\alpha \geq \pi/2$, then $s \leq 0$, indicating that the central region appears on the other side of $OA$ compared to the case $\alpha <\pi/2$.

As the moduli of convex quadrilaterals with three right angles has two parameters $\alpha$ and $d$, for a fixed angle $\alpha$, we have a one-parameter family of right-angled quadrilaterals, with isometry type determined by $d$, or alternatively, by $s$.

The linear relation (\ref{f:combining}) between the values $t$ and $s$ 
says the following:
\begin{proposition}\label{central}
For a fixed angle $\alpha$, as the value of $d$ varies, the central region of the quadrilateral  changes its shape via Euclidean homotheties, centered at the origin $O$, where the scaling is given by the value $s>0$ with $d = \log \frac{1+s}{1-s}$.    
\end{proposition}

Now we come back to the hexagonal setting where three pentagons, or equivalently, six quadrilaterals, are   combined.  Type I hexagon is represented in Figure \ref{nh1}.

\begin{figure}[htbp]

\centering
 \psfrag{L1}{\small $L_1$}
  \psfrag{L2}{\small $L_2$}
   \psfrag{L3}{\small $L_3$}
      \psfrag{s}{\small $s$}
  \psfrag{p1}{\small $\alpha_1$} 
 \psfrag{p2}{\small $\alpha_2$} 
  \psfrag{p3}{\small $\alpha_3$}
  \psfrag{a1}{\small $\ell_1$} 
 \psfrag{a2}{\small $\ell_2$}
 \psfrag{a3}{\small $\ell_3$}  
  \psfrag{long}{\small long} 
  \psfrag{short}{\small short}
\includegraphics[width=6cm]{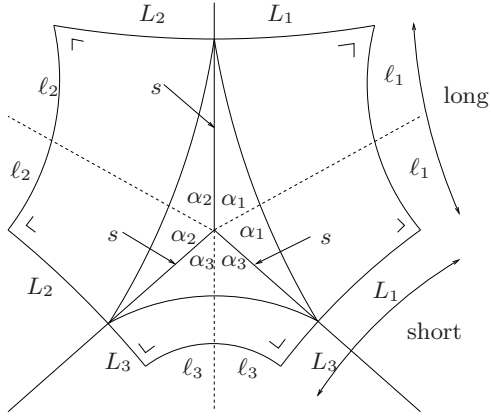}
\caption{The central region has one  $s$, which is the common Euclidean length of the three branches of the tripod $T$.} \label{nh1}
\end{figure}

 The hyperbolic lengths $\ell_i, L_i \,\,\, (i =1,2,3)$ are all functions of $d$ only, {\it provided the angle moduli $(\alpha_1, \alpha_2, \alpha_3)$ are fixed.}

  Equation (\ref{e1}) says that 
\[
\cosh \ell_i = \sin \alpha_i \cosh d. 
\]
By taking the ratio of the equality for $i$ and $j$, we have 
\[
\frac{ \cosh \ell_i}{\cosh \ell_j} = \frac{\sin \alpha_i}{\sin \alpha_j},
\]
a quantity independent of $d$.  Hence  for different values of $d$ and $d'$,
\[
\frac{\cosh \ell_i(d')}{\cosh \ell_j(d')} = \frac{\cosh \ell_i (d)}{\cosh \ell_j(d)}
\]
for $1 \leq i, j \leq 3$.

This equality implies the following:

\begin{proposition} Suppose the angles $\alpha_1, \alpha_2, \alpha_3$ are fixed.  
As the value of $d$ varies, the ratios among $\cosh \ell_1(d)$, $\cosh \ell_2(d)$ and $\cosh \ell_3(d)$ of the right-angled hexagons remain  invariant. \end{proposition}

In other words, the ratios of weighted lengths (where taking the weighted length means replacing it by the hyperbolic cosine of that length) are invariant.

 We concentrate on the deformations such that the long edges $s_i$ get longer, or, equivalently, the scale of the central region gets larger. We can call  these deformations {\it forward} deformations.  The forward deformation of $H$  makes {\it long edges longer and short edges shorter.} We recall that the angles $\alpha_i$ are constant for the forward deformations. In the limit of such a deformation, as $L_i \rightarrow 0$, one obtains ideal triangles. 
 
 Let us note by the way that unlike the deformation studied in the present paper, the deformations studied  by Thurston in the paper \cite{Thurston} and by Papadopoulos-Th\'eret in the paper \cite{PT2} shrink the horocyclic/hypercyclic foliations and dilate the geodesic foliations.

The convergence $L_i \rightarrow 0$ is studied using the formula
\[
\tanh L = \frac{\cos \alpha}{\sin \alpha} \frac{1-s^2}{2s}
\]  
regardless of the cases $\alpha < \pi/2$ or $\alpha \geq \pi/2$.
Note that
$L\rightarrow 0$ and $s \rightarrow  1$
are equivalent, and that $s \rightarrow 1$ means that the long edge is being pushed to the geometric boundary of the Poincar\'e disc. Hence the one-parameter family of forward deformations provides a canonical path for the 
right-angled hexagon asymptotically converging to an ideal triangle.  Also, in terms of elongating the long side by a factor $K$ (here $K = K(d)$) we can see from an equality obtained from the trigonometric formulae (\ref{e1}) and (\ref{e5}), namely, 
\[
\tanh^2 L(K) = \frac{\cos^2 \alpha}{\cosh^2 (K \ell) - \sin^2 \alpha},
\]
that $L(K)$ approaches zero as $K\to \infty$ along the forward deformation.

Let us consider the special ``symmetric" case where $\alpha = \pi / 3$, as this is the case treated in detail in \cite{PT2}.  Then six copies of the trirectangular quadrilateral (up to orientation) placed together form a right-angled hexagon with ${\mathbb Z}_3$-rotational symmetry, which has side lengths $2L$ and $2 \ell$ appearing alternately.  

Inserting the values $\alpha = \pi/3$ and $\ell = 0$ in the formula (\ref{e1}),  we obtain
\[
\frac{1}{\sqrt{3}/2} = \frac{\cosh d}{1}.
\]
Since $\cosh d = 2s/(1-s^2)$, we get $s = 2 - \sqrt{3}$.  This value of $s$ is 
obtained by letting $\ell$ approach zero, or equivalently by letting the right-angled hexagon converge to an ideal triangle via ``backward" deformations.  The value of $t$ then is $(2-\sqrt{3})^2$.   When 
deformed in the forward direction, the hexagon also converges to the ideal triangle, this time with $t = 
2 - \sqrt{3}$  and $s= (2-\sqrt{3})^2$.  Note that the forward deformation is defined for all $K \geq 1$, 
while in general the backward deformation is not.

%
%
%
%
%
%
%
%
%
%
%
%
%

 \section{A geometric lemma}
Let us recall a few elements from \cite{PT2}.

A map $f$ between two metric spaces is said to be \textsl{contracting} if $\hbox{Lip}(f)<1$, and \textsl{weakly} contracting if $\hbox{Lip}(f)\leq1$

If $f$ is a self-map of class $C^{0,1}$ of a convex domain $\Omega$ of the hyperbolic plane, then one can compute the norm of its differential at each point $x$ of $\Omega$,  

$$ ||(df)_{x}||=\sup_{V\in T_{x}\Omega\setminus\{0\}}\frac{||(df)_{x} (V)||}{||V||}.$$
Setting 
\[
||df||=\sup_{x\in\Omega}||(df)_{x}||,
\]
we have the following

\begin{proposition} The quantity  $\Vert df \Vert$ provides an upper bound for the global Lipschitz constant. In other words, we have
\[
\hbox{Lip}(f)\leq \Vert df \Vert.
\]
\end{proposition}

In what follows, we  consider maps between hyperbolic surfaces. The following statement will be used later 
in relating two hyperbolic surfaces.

\begin{lemma} \label{lip.const}
Suppose there exist two pairs of orthogonal 
foliations $(F_1,G_1)$ and $(F_2,G_2)$ which are preserved by $f$ (that is, $f$ sends any leaf of 
$F_1$ onto a leaf of $F_2$ and any leaf of $G_1$ onto a leaf of $G_2$).  If $f$ is 
$K$-Lipschitz along the leaves of $F_1$  and contracting  along the leaves of  $G_1$, then  $f$ is $K$-Lipschitz (that is, $\hbox{Lip}(f) \leq K$). Furthermore, if the map $f$ stretches the leaves of $F_1$ by  a constant factor $K$ and if there exists 
an open geodesic segment contained in a leaf of $F_1$ which is sent to a geodesic segment 
contained in a leaf of $F_2$ (with length expanded by $K$), $K$ is the best Lipschitz constant for $f$.    
\end{lemma}

\begin{proof}
The proof follows from an argument in \cite{PT2} (p. 65-66) which we reproduce here.  Consider the 
coordinate system $(\alpha, \beta)$ defined by the orthogonal grid formed by the leaves of $F_1$ and 
$G_1$, where $\alpha$ and $\beta$ are suitably chosen so that  the coordinate vector fields $
(\frac{\partial}{\partial \alpha}, \frac{\partial}{\partial \beta} )$ are an orthonormal basis for each tangent 
plane, and are linearly sent to an orthogonal basis by the differential $df$.   

If $f=(f_\alpha, f_\beta)$, then 
\[
df_{(\alpha, \beta)} = \frac{\partial f_\alpha}{\partial \alpha} d \alpha + \frac{\partial f_\beta}{\partial \beta} d \beta.
\]
Let $V = V_\alpha \frac{\partial }{\partial \alpha} +  V_\beta \frac{\partial }{\partial \beta}$ a vector field
with $\|V\|=1$.  The norm of the differential $\|df\|$ is computed as follows.  First, we have 
\begin{eqnarray*}
\|df_{(\alpha, \beta)} (V) \|^2  & = &  \Big\| \frac{\partial f_\alpha}{\partial \alpha} V_\alpha \frac{\partial}{\partial \alpha} + \frac{\partial f_\beta}{\partial \beta} V_\beta \frac{\partial}{\partial \beta} \Big\|^2 \\
 & = &  \Big( \frac{\partial f_\alpha}{\partial \alpha} V_\alpha \Big)^2 +  \Big(\frac{\partial f_\beta}{\partial \beta} V_\beta \Big)^2. 
 \end{eqnarray*}
From this, it follows that 
\[
\|df_{(\alpha, \beta)} (V) \| \leq \max_\Omega  \Big\{  \Big|  \frac{\partial f_\alpha}{\partial \alpha}  \Big|,  \Big|  \frac{\partial f_\beta}{\partial \beta} \Big| \Big\} \|V\|
\]
which in turn gives 
\[
\|df\|  \leq \max_\Omega  \Big\{  \Big|  \frac{\partial f_\alpha}{\partial \alpha}  \Big|,  \Big|  \frac{\partial f_\beta}{\partial \beta} \Big| \Big\}.
\]
Now since the map $f$ is $K$-Lipschitz along the leaves of $F_1$, we have $\Big| \frac{\partial f_\alpha}
{\partial \alpha} \Big| \leq  K $  at any point of $\Omega$.  On the other hand  since the map $f$ is contracting 
along the leaves of $G_1$,   we have $\Big| \frac{\partial f_\beta}
{\partial \beta} \Big| < 1 $ at any point of $\Omega$.  Thus we have   $\hbox{Lip}(f) \leq \|df\| \leq K$.  

\end{proof}

\section{Deforming right-angled hexagons}

Given three non-negative numbers $(\ell_1,\ell_2,  \ell_3)$, let $H$ be a right-angled hexagon in the hyperbolic plane with pairwise non-consecutive side lengths $(2\ell_1,2\ell_2,2\ell_3)$.  Such a hexagon is uniquely determined up to isometry by these three sides.  Let us now consider these sides to be the {\it long} edges of $H$. We denote the lengths of  
the diametrically facing three pairwise non-consecutive sides (the {\it short} edges) by $\lambda_1, \lambda_2,\lambda_3$.
If some $\ell_i$ is equal to $0$, we consider that the corresponding side is at infinity. For  any $K>1$, we denote by
$H_K$  the right-angled hexagon with long edges of lengths $(2  \ell_1 (K),2  \ell_2 (K),2 \ell_3(K))$, with $\ell_i(K)$, for $i=1,2,3$, determined by
\[K=\frac{\cosh \ell_i(K)}{\cosh \ell_i}.
\]

For $i=1,2,3$, we define 
\begin{equation}\label{eq:k_i}
k_i=\frac{\ell_i(K)}{\ell_i}
\end{equation}
and we set
\[k=\max_i\{k_i\}.
\]

Note that $k$ is determined by $H$ and $K$. We also note that from its definition, $k$ is an increasing function of $K$ and that when $K$ varies from $1$ to $\infty$, so does $k$.

In the rest of this section, we construct a $k$-Lipschitz map $f_k: H \to H_K$ which 
 is Lipschitz extremal in its homotopy class in the sense that the Lipschitz constant of any other map between the two given right-angled hexagons sending each edge of $H$ to the corresponding edge of $H_K$ is at least equal to $k$. Making $K$ (or, equivalently, $k$) vary from 1 to $\infty$, we obtain a family of marked right-angled hexagons which is a geodesic for the Lipschitz metric on the Teichm\"uller space of right-angled hyperbolic hexagons -- a natural analogue of the metric defined above in (\ref{Lip}). By gluing hexagons along their sides, we obtain families of geodesics for the  arc metric for the Teichm\"uller space of surfaces with nonempty boundary (and with variable lengths of the boundary components). Gluing surfaces with boundary along their boundary components and  assembling the maps between them, we obtain geodesics for the Thurston metric on the Teichm\"uller space of surfaces without boundary (possibly with punctures).

The maps $f_k$ that we construct from $H$ to $H_K$ preserve the two pairs of orthogonal partial  foliations  $(F,G)$ 
and   $(F^K,G^K)$ of  $H$ and $H_K$ respectively that we constructed in \S \ref{s:geo}.

The $k$-Lipschitz map $f_k: H\to H_K$  satisfies the following conditions: 
\begin{enumerate}
\item $f_k$ stretches  each long edge $\ell_i$ of $H$ by the factor $k_i >1$.
\item $f_k$ sends each leaf of $F$ to a leaf of $F^K$  affinely (with respect to the natural arc-length parametrization)  and it sends each leaf of $G$ to a leaf of $G^K$.
\item  The central region is sent to the central region.
\end{enumerate}

Using the coordinates we established on each foliated pentagon $P_i$ (the foliated rectangle with part of the foliated central region), the map $f_k$ is defined by sending the point on $H$ represented by $(u_i, v_i)$ for some $i$ to the point in $H_K$ represented by the 
same 
coordinates $(u_i, v_i)$. This map is clearly a homeomorphism between $H$ and $H_K$. It sends the leaves of $F$ and $G$ on $H$ to those on $H_K$, and the central region  to the central region.
We will show a few  properties of the map $f_k$, concluding that this map is Lipschitz extremal between the pair of right-angled hexagons.

\begin{theorem}\label{th:K}
The map $f_k$ is an extremal $k$-Lipschitz map from $H$ to $H_K$ for all $K>1$.
\end{theorem}

We first recall another result in hyperbolic geometry stated as a lemma:

\begin{lemma}\label{lem:part}
A segment of the hypercycle $F(u_i)$ with $0 \leq u_i \leq 1$, which projects onto a  
geodesic segment of length $r>0$ on the long edge $s_i$ via the nearest point projection map $\pi_i: P_i \rightarrow s_i$, has hyperbolic length $r(u_i)$, which is given explicitly by 
\[
r(u_i) = \cosh (u_i L_i) r.
\]
\end{lemma}

\begin{proof}   \begin{figure}[!hbp] 
\centering
 \psfrag{ui}{\small $u_i$}
  \psfrag{1ui}{\small $1-u_i$}
   \psfrag{rui}{\small $r(u_i)$}
    \psfrag{r}{\small $r$}
  \psfrag{d}{\small $d$}
   \psfrag{O}{\small $O$}
     \psfrag{si}{\small $s_i$}
\psfrag{Li}{\small $L_i$}
  \psfrag{Fui}{\small $F(u_i)$}
\includegraphics[width=0.60\linewidth]{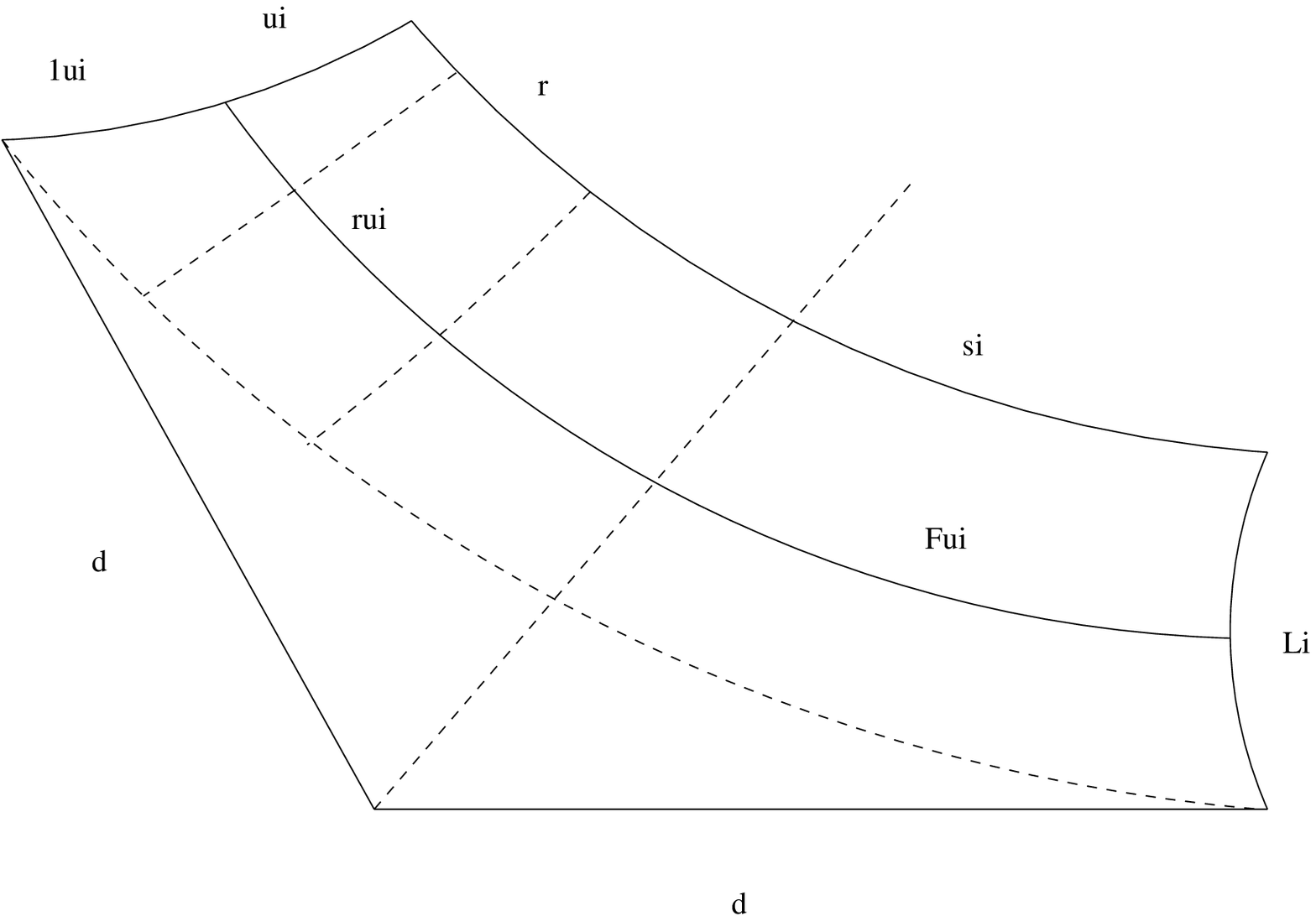}   
 \caption{{\small}}  \label{fig-new}  
\end{figure}
We use the following elementary fact from the hyperbolic geometry of the upper-half plane model: The ray through the origin of slope $m = (\sinh u_i L)^{-1}$ is a set of points which are equidistant of $u_i L$ from the geodesic represented by the positive $y$-axis. (See Figure \ref{fig-new}.) Then the hypercycle is represented by an oblique ray.  The fibers of the nearest point projection map from the oblique ray to the positive $y$-axis are (subsets of) concentric circles centered at the origin.  Hence in order to prove the lemma it suffice to compute the hyperbolic stretching factor between the geodesic line segment between $(0, t_1)$ and $(0, t_2)$ with $t_2 > t_1$, and the hypercycle segment 
represented by the line segment between $(\frac{1}{\sqrt{1 + m^2}}t_1, \frac{m}{\sqrt{1 + m^2}}t_1)$ and $(\frac{1}{\sqrt{1 + m^2}}t_2, \frac{m}{\sqrt{1 + m^2}}t_2)$.  We carry out this computation by comparing the hyperbolic norm of two paths $\sigma(t) = (0, t)$ and $\rho(t) = ((\frac{1}{\sqrt{1 + m^2}}t, \frac{m}{\sqrt{1 + m^2}}t)$ defined over $[t_1, t_2]$.

  Recall that the length element in the upper-half plane is given by
\[
ds^2 = \frac{dx^2 + dy^2}{y^2}.
\]
Then the hyperbolic norms of the tangent vectors are given by
\[
\| \sigma'(t) \| = \frac1t, \quad \| \rho'(t) \| = \frac1t \frac{\sqrt{1+m^2}}{m}. 
\]
By taking the ratio and replacing $m$ by $(\sinh u_i L)^{-1}$, it follows that 
\[
\frac{\| \rho'(t) \|} {\| \sigma'(t) \|} = \frac{\sqrt{1+m^2}}{m} = \cosh u_i L,
\]
a number independent of $t$.  This proves that the nearest point projection induces a stretching between the geodesic and the hypercycle by the factor $\cosh u_i L$.
\end{proof}

\begin{proof}[Proof of Theorem \ref{th:K}]  
Lemma \ref{lem:part} implies that each leaf $F_i(u_i)$ of the foliations $F_i$,  for $0 \leq u_i \leq 1$,   is obtained by  
stretching the geodesic segment $s_i = F_i(0)$ by the constant factor $\cosh (u_i L_i)$ and that each leaf $F_i^K(u_i)$ of the foliations $F_i^K$,  for $0 \leq u_i \leq 1$,   is obtained by  
stretching the geodesic segment $s_i^K = F_i^K(0)$ by the constant factor $\cosh (u_i L^K_i)$
where $L_i^K$ is the length corresponding to $L_i$ in the right-angled hexagon $H_K$. 
 We note that $L^K_i < L_i$. 
 
  Since the map $f_k$ sends the leaf $F_i(u_i)$ to $F_i^K(u_i)$, and the long edge $s_i$ of length $2\ell_i$ to the long edge $s_i^K$ of length $2k_i \ell_i$,  the Lipschitz constant of $f_k$ restricted to the leaf $F_i(u_i)$ is
given by
\[
{\rm Lip}(f_k \big|_{F_i(u_i)}) = k_i \frac{\cosh (u_i L_i)}{\cosh (u_i L_i^K)} \leq k 
\]
where the inequality follows from $L^K_i < L_i$.  

Concerning the foliation $G_i$, whose leaves are perpendicular to the leaves of $F_i$,  for $0 \leq v_i \leq 2$, each leaf $G_i(v_i)$ of length $L_i$, namely the part of the leaf in the quadrilateral  $ P_i \setminus C$,  is sent to the part of the leaf $G_i^K(v_i)$ in the quadrilateral $ P_i^K \setminus C^K$ of length $L_i^K$. Thus the Lipschitz constant is
\[
{\rm Lip}\Big( f_k \big|_{G_i(v_i)} \Big) = \frac{L_i^K}{L_i} < 1 < k. 
\] 
 
Our next task is to obtain a control on the Lipschitz constant  of $f_k$ in the central region of the hexagon.  We divide the central region $C$ into the three sectors $\{P_i \cap C\}_{i=1,2,3}$, whose interface is the geodesic tripod $T$ centered at the origin $O$ of the Poincar\'e disc, as discussed in \S \ref{s:geo}.

Let $d(K)$ be the edge length of the tripod in the right-angled hexagon $H_K$, spanning the central region $C_K$.  From 
the preceding discussion, we have $d(K) > d=d(1)$.  We claim now that the map $f_k$ restricted to each tripod edge has 
Lipschitz constant bounded above by $k$. 

 Since (by definition) the map $f_k$ is an affine map on each edge, to show the 
Lipschitz constant bound, it suffices to obtain an estimate at one point.  We do so at the point $A_1$.  Let $\triangle s$ be 
an infinitesimal interval along the edge $OA_1$ at $A_1$. By abuse of notation, we denote by $\triangle s$ the hyperbolic length of that interval. Note that the edge $OA_1$ and the hypercycle $F_2(1, v_2) \,
\,\, (0 \leq v_2 \leq 2) $ are tangent at $A_1$. Thus  $\triangle s$ corresponds to $\tilde{\triangle} s \subset s_2$ which is a multiple $
\triangle s/ \cosh L_1$ of $\triangle s$ on $s_2$. This follows from  Lemma \ref{lem:part}  via the nearest point projection 
$OA_1 \rightarrow s_2$ onto the long edge $s_2$.  The long edge $s_2$ is stretched by $f_k$ by the factor $k_2$, and 
the interval $\tilde{\triangle} s$  is mapped to $ \hat{\triangle} s \subset s_2^K$. Thus the length element $\triangle s$ on 
$OA_1$ now corresponds to 
$\hat{\triangle} s = k_2 \, \triangle s / \cosh L_1$ on $s_2^K$.  Finally, via the canonical correspondence between the long 
edge $s_2^K$ and the hypercycle $F_2^K(1, v_2) \,\,\, (0 \leq v_2 \leq 2)$ along the fibers of the  nearest point projection 
$OA_1^K \rightarrow s_2^K$, $\triangle s$ corresponds to an infinitesimal interval $\overline{\triangle} s \subset F_2^K(1, 
v_2) $ a multiple of $\triangle s$:
\[
\frac{1}{\cosh L_1} k_2 \cosh L_1^K \triangle s.
\]
The hypercycle  $F_2^K(1, v_2) \,\,\, (0 \leq v_2 \leq 2)$ and the tripod edge $OA_1^K$ are again tangent at $A_1^K$. Using the fact that $L_1^K < L_1$,
we conclude that the Lipschitz constant of $f_k$ restricted to the edge $OA_1$ at $A_1$ is equal to 
\[
\frac{\overline{\triangle} s}{\triangle s} = \frac{ \cosh {L_1^K}}{\cosh L_1}  k  < k.
\]
Therefore, we conclude that the Lipschitz constant of $f_k$ restricted to the edge $OA_1$ is strictly less than $k$.
 
 The region $P_i \cap C$ is foliated by the hypercycles equidistant from the long edge $s_i$, which are the leaves of 
$F_i(u_i)$ for $1\leq u_i \leq 2$, whose endpoints are on the two edges $\overline{OA_j}$ and $\overline{OA_k}$ $(j, k 
\neq i)$ of the tripod bounding the pentagon $P_i$.  Hence $P_i \cap C$ is bounded by the two edges and the boundary 
hypercycle $F_i(1)$ of the central region $P_i \cap C$ equidistant from the long edge $s_i$ by $L_i$.  

We define a new function of $u_i$, 
\[
\tilde{u}_i (u_i) = \frac{\mbox{the hyperbolic distance between $F_i(1)$ and $F_i(u_i)$}}{\mbox{the hyperbolic distance between $F_i(1)$ and the vertex $O$}}
\]
for $1 \leq u_i \leq 2$.  Note that $\tilde{u}_i (1) = 0$ and $\tilde{u}_i (2) = 1$.
We also define $\tilde{u}^K_i(u_i)$ to be the corresponding function for the forward-deformed hexagon $H_K$, which is distinct from $\tilde{u}_i(u_i)$.

We claim that 
\[
{\rm Lip}\Big( f_k \big|_{F_i(u_i)} \Big) = k_i \frac{\cosh (\tilde{u}_i L_i)}{\cosh (\tilde{u}^K_i L_i^K)} < k. 
\]  
This follows from the facts that $L_i^K < L_i$ and that $\tilde{u}_i^K (u_i) < \tilde{u}_i(u_i)$ for $1\leq u_i \leq 2$. The first inequality follows from the 
properties of the forward deformation $f_K$ already seen.  For the second inequality, note that the hypercycles $\{F^K_i(u_i)\}$ 
are more tightly packed in $P_i^K \cap C_K$ near the hypercycle $F_i^K(1)$ than  the hypercycles $\{F^K_i(u_i)\}$ in $P_i \cap C$ near $F_i (K)$, which in turn implies 
$\tilde{u}_i^K (u_i) < \tilde{u}_i(u_i)$.

Now we study the Lipschitz constant of $f_k$ restricted to the leaves of  $G_i$ in the central region $P_i \cap C$. We recall that the leaves $F_i(u_i)$ are parameterized  
proportionally to the arc length of the pair of the tripod edges $\overline{OA_j}$ and $\overline{OA_k}$ $(j, k \neq i)$ of equal length $d$, as $u_i$ varies in $[1, 2]$.  

Let $\delta_i(u_i, v_i)$ be the hyperbolic distance between the point $G_i(v_i) \cap F_i(u_i)$ and the long
edge $s_i$.  First note that the function $\delta_i(u_i, v_i)$ is constant in $v_i$ where it is defined. As $v_i$ 
varies over $[0,2]$, the range of $\delta_i$ varies, and it takes its maximal value for $v_i 
=1$ when the endpoint of the leaf $G_i(1)$ is at the center $O$. 

We define a function $w_i(u_i)$ in $u_i \in [1,2]$ by restricting the function $\delta_i$ to the tripod edge $\overline{OA_j}$ (or equivalently to $ \overline{OA_k}$) and subtracting the constant $L_i$ from it.  Namely, the value of $w(u_i)$ is the distance between the point $(u_i, v_i)$ on $\overline{OA_j}$    (or equivalently to $ \overline{OA_k}$)  and the boundary hypercycle of the central region $P_i \cap C$.  
By 
observing how the leaves $F_i(u_i)$ intersect with the tripod edge $\overline{OA_j}$ modeled on the Poincar\'e disc, with the center $O$ of $H$ identified with the origin, we see that the term $\frac{d w_i}{d u_i}$ is equal to $d \cos \theta_i(u_i)$ 
where $\theta_i$ is the angle between the edge $OA_j$ and the leaf $G_i(v_i)$. Note  that $\theta(1)=\pi/2$ and $\theta_i(u_i)$  monotonically decreases as $u_i$ increases, and goes down to 
the value $\alpha_i = \theta_i(2) < \pi/2$.
  
It then follows that the derivative of $w_i$ in $u_i$ is monotonically increasing and that 
\[
\lim_{u_i \rightarrow 1} \frac{d w_i}{d u_i}  = 0  \mbox{ and }  \lim_{u_i \rightarrow 2} \frac{d w_i }{d u_i} = d \cos \alpha_i .
\]
The latter limit occurs at the vertex $O$.  In other words, the function $w_i$ is convex in $u_i$.

Defining the function $w_i^K$ for $H_K$ analogously, we have similarly:
\[
\lim_{u_i \rightarrow 1} \frac{d w_i^K}{d u_i}  = 0  \mbox{ and }  \lim_{u_i \rightarrow 2} \frac{d w_i^K }{d u_i} = d(K) \cos \alpha_i .
\]

We now claim that the function 
\[
 R_i(u_i) =\Big( \frac{d w_i^K}{d u_i} \Big) /   \Big( \frac{d w_i}{d u_i} \Big)
\]
is increasing in $u_i \in [1, 2]$ and bounded above by:
\[
\lim_{u_i \rightarrow 2} \Big(\frac{d w_i^K(u_i)}{d u_i} \Big)/\Big(\frac{d w_i(u_i)}{d u_i}\Big) = \frac{d(K)}{d} < k.
\] 
Namely, the claim says that the biggest stretch by $f_k$ along 
the leaves of $G_i$ occurs at the center of the hexagon on the leaf $G_i(1)$.  The upper bound $ d(K)/d$ is 
obtained from the ratio between the separation distances of leaves for $F^K$ and $F$ near the center of 
the hexagon, identified with the origin of the Poincar\'e disc.

This follows from the observation that for a given value of $u_i \in (1, 2)$, we have $\theta_i^K(u_i) > \theta_i(u_i)$, which follows from the comparison of the behavior of the Poincar\'e metric of the disc between the two Euclidean-homothetic regions $C$ and $C_K$.  Geometrically, when looking at the distributions of the leaves of  $F_i(u_i)$ and $F_i^K(u_i)$ on  $C$ and on $C_k$, which are
mutually Euclidean-homothetic in the Poincar\'e disc,  the leaves for $F_i^K$ are more tightly packed  than 
the leaves of $F_i$  near $u_i =1$ (Figure \ref{region}).   
 \begin{figure}[!hbp] 
\centering
 \psfrag{f}{\small $f_K$}
  \psfrag{P}{\small $P_i\cap C$}
   \psfrag{Pk}{\small $P_i^K\cap C_K$}
    \psfrag{O}{\small $O$}
  \psfrag{F}{\small $F_i(u_i)$}
   \psfrag{Fk}{\small $F^k(u_i)$}
\psfrag{ui}{\small $u_i$}
  \psfrag{u1}{\small $(1-u_i)$}
\includegraphics[width=0.80\linewidth]{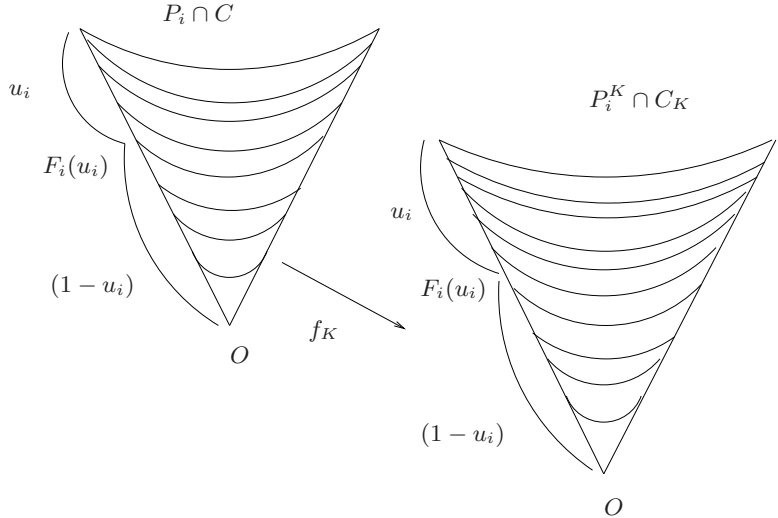}    
\caption{\small {The foliations in the central region of the hexagon.}}   \label{region}  
\end{figure} 
Hence  for a  particular  pair of leaves $F_i(u_i)$ and $F_i^K(u_i) = f_K[F_i(u_i)]$, we have the inequality $\theta_i^K(u_i) > \theta_i(u_i)$. This comparison in turn shows that as a convex function of $u_i$, $w_i(u_i)$ is more convex than the convex function $\frac{d}{d(K)}w_i^K(u_i)$. 
This implies the claim about $R_i(u_i)\leq\frac{d(K)}{d}$.

This in turn says that the rate of change
\[
\frac{d w_i^K}{d w_i} = \Big(\frac{d w_i^K}{d u_i} \Big)/\Big(\frac{d w_i}{d u_i}\Big) 
\]
is bounded above by $k_i$, that is, since the map $f_k$ send $P_i \cap C$ onto $P_i^K \cap C^K$, the stretching factor of the leaves of $G_i$ is strictly less than $k_i$.

Combining these results, we conclude that for $1 \leq u_i \leq 2$, 
\[
{\rm Lip} \Big(f_k \big|_{G_i(v_i)} \Big) < \frac{d(K)}{d}  < k_i.
\] 

Now the map $f_k$ sends the central region $C$ to $C_K$ and each open set $P_i$ to $P_i^K$ with 
norm $\|df_K\|$ of the differential bounded by $k$ and it sends the tripod $T$ to the tripod $T_K$ 
with constant stretch ratio of $d(K)/d< k$.  These three regions $P_i$ can be glued 
along the tripod $T$, and the stretching Lipschitz constant remains less than $k$, proving that $f_k:H \rightarrow H_K$ 
is $k$-Lipschitz. 
 
\end{proof}

\section{ideal triangulation of a surface and the canonical deformations}

So far we have considered a single right-angled hexagon, and its forward deformations. We consider now a hyperbolic surface $S$ with boundary which is partitioned 
by a maximal sub-system of disjoint geodesic arcs into right-angled hexagons. Such a partition is usually called an ideal triangulation $T$ on $S$ (see e.g. \cite{Luo}.)  We 
denote this combinatorial data by $(S, T)$.  The gluing edges of the hexagons whose union constitutes the surface $S$ are considered as the long edges,
and the edges appearing as part of the boundary components of $S$ are considered as the short edges.  

We use the notation of \S \ref{s:geo}. By stretching each long edge $s_i$ of a hexagon by a factor $k_i>1$ satisfying Equation (\ref{eq:k_i}), the canonical deformations between hexagons define a map, which we still denote $f_k: S \rightarrow S_K$, from the hyperbolic surface with boundary $S$ to the hyperbolic surface with boundary $S_K$ obtained by gluing the image hexagons. Here, $k$ is the maximum among the stretching factors $k_i$, with $i$ indexing the long edges of all the right-angled hexagons. This map $f_k$ makes each boundary component shorter.   

\begin{proposition}
The Lipschitz map $f_k: S \rightarrow S_K$ is an extremal $k$-Lipschitz map. That is, this map is a $k$-Lipschitz homeomorphism, and for any $k' < k$, there is no $k'$-Lipschitz homeomorphism from $S$ to $S_K$ in the same homotopy class preserving the boundary.
\end{proposition}  

\begin{proof}
We use Lemma \ref{lip.const} and we note that when two hexagons $H^\alpha$ and $H^\beta$ are 
glued along a long edge $e$, the foliations $F^\alpha, G^\alpha$ and $F^\beta, G^\beta$ are glued 
together to form a  new orthogonal pair of foliations on the union of the image hexagons. The constructed maps between the hexagons satisfy the properties required by the proposition, and the hexagons glue together as well. 
\end{proof}

The following theorems are proved in the same way as Theorem 7.3 of \cite{PT2}; we do not repeat the arguments here. In the following statements, we take the point of view where each map $f_k: S \rightarrow S_K$ is an element, which we also denote by $(S,f_k)$, of the Teichm\"uller space of $S$, that is, a marked hyperbolic structure.

\begin{theorem}
By letting $K$ vary from 1 to $\infty$, we get a family of maps $f_k: S=S_1 \rightarrow S_K$ which, as a path in Teichm\"uller space,  is a geodesic for the arc metric and for the Lipschitz metric on $\mathcal{T}(S)$.
\end{theorem}

\begin{theorem}
The Lipschitz and the arc metrics on  $\mathcal{T}(S)$ coincide on the path $(S,f_k)$, $k\geq 1$. More precisely, we have, for any $K_1\leq K_2$, 
\[d([(S,f_{K_{1}})], [(S,f_{K_{2}})])= L([(S,f_{K_{1}})], [(S,f_{K_{2}})])
\]
where $d$ and $L$ are the arc and Lipchitz metrics respectively. 
\end{theorem}  

By gluing surfaces with boundary along the totally geodesic boundary components, we also get geodesics for the  Lipschitz metric for surfaces without boundary, as in \cite{PT2}. 

As mentioned in the introduction, the resulting geodesics for a surface $S$ with boundary are  distinct from the geodesics constructed by Thurston using the stretch line construction where the ideal triangles 
spiral along the boundary components of $S$.  For surfaces with more than one boundary component, this follows from the fact that Thurston's path causes the lengths 
of all boundary components change by the same factor, whereas our deformation causes the length of each boundary components  to change at different 
factors in general.   

There is a relation between our coordinates on right-angled hexagons and coordinates constructed by Luo in \cite{Luo} for the Teichm\"uller  space of a surface with boundary.  Luo calls his coordinate functions  {\it radius coordinates}. In our context, they correspond to the sum of the lengths $L_i$.  To be more precise,  consider two right-angled hexagons $H^\alpha$ and $H^\beta$, sharing a long edge of the same length, which we call the edge $e_{\alpha\beta}$. The  edge $e_{\alpha\beta}$  borders two strips, each foliated by leaves equidistant from the edge $e_{\alpha\beta}$ bordered on the other end by the sides of the central regions in $H^\alpha$ and $H^\beta$.  In our notation, the widths of the strips, namely the distances from the shared edge to the respective central regions, are called $L^\alpha(e_{\alpha\beta})$ and $L^\beta(e_{\alpha\beta})$. The radius coordinate of Luo
is equal to 
\[
z(e_{\alpha\beta}) = \frac{L^\alpha(e_{\alpha\beta})+L^\beta(e_{\alpha\beta})}{2}. 
\]  
By setting $E$ to be the set of long edges of an ideal triangulation $T$ with respect to a hyperbolic metric on a surface $S$ with totally geodesic boundary, we can consider the set of functionals
$z: E \rightarrow  {\mathbb R}$ 
on the space of hyperbolic metrics of the surface with totally geodesic boundary.  Luo calls $z(e)$ the radius invariant of $e$, and $z$ the radius coordinate system of $(S, T)$.

In \cite{Luo},  the following is shown: 

\begin{theorem} Given an ideal triangulation $(S, T)$ on a compact surface $S$ with boundary, each hyperbolic metric with totally geodesic boundary on $S$ is determined 
up to isotopy by its radius coordinates.  Furthermore, the image of the map $z$ is a convex polytope satisfying the following two properties: for each fundamental edge cycle $e_1, ..., e_k$
\[
\sum_{j=1}^k  z(e_j) > 0
\]
and for each boundary edge cycle $e_1, ... e_n$ corresponding to the boundary component of length $\ell$, 
\[
\sum_{i=1}^n  z(e_i) = \ell.
\]
\end{theorem}

 Finally, we mention that D. Alessandrini and V. Disarlo informed us that they have work in progress in which they define ``generalized stretch lines" for surfaces with boundary and they prove that every two hyperbolic structures can be joined by a geodesic
segment which is a finite concatenation of such
generalized stretch lines, cf. \cite{AD}.

\medskip

\noindent \emph{Acknowledgements} The authors are grateful to the referee for his thorough reading and his valuable corrections.

 \end{document}